\documentclass{article}
\usepackage{mathtools}
\usepackage{amssymb}
\usepackage{amsmath}
\usepackage{amsthm}
\usepackage{verbatim}
\usepackage{color}
\usepackage{url}
 \usepackage[all]{xy}
\usepackage{fancyhdr}
\usepackage{graphicx,amsmath,amsfonts,latexsym,amssymb,amsthm,color}
\usepackage{latexsym}
\usepackage{verbatim}
\usepackage{enumerate}
\usepackage{enumitem}

\usepackage[a4paper]{geometry}
\definecolor{blau}{rgb}{0.05,0.2,0.7}
\definecolor{auchblau}{rgb}{0.03,0.3,0.7}
\usepackage[pdftex,linkbordercolor={0.8 0.5 0.2}]{hyperref} 
\usepackage{xcolor}

\newtheorem{theorem}{Theorem}[section]
\newtheorem*{theorem*}{Theorem}
\newtheorem*{lemma*}{Lemma}
\newtheorem{proposition}[theorem]{Proposition}

\newtheorem{corollary}[theorem]{Corollary}
\theoremstyle{definition}

\newtheorem*{definition*}{Definition}

\newtheorem*{remark*}{Remark}

\numberwithin{equation}{section}

\newcommand{\R}{\mathbb{R}}
\newcommand{\C}{\mathbb{C}}

\renewcommand{\Re}{\operatorname{Re}}
\renewcommand{\Im}{\operatorname{Im}}

\DeclareMathOperator{\supp}{supp}

\newcommand{\e}{\mathrm{e}}
\newcommand{\dx}{\mathrm{d}}

\newcommand{\bS}{\mathbb{S}}

\numberwithin{equation}{section}
 
\title{A Pöschl-Teller Approximation to the Wave Equation Outside a Torus}
\author{Ronald Quirchmayr \quad and \quad Alden Waters\thanks{Institute for Analysis, Department of Mathematics, Leibniz University Hannover, Welfengarten 1, 30167 Hannover, Germany, alden.waters@math.uni-hannover.de} \hspace{0.05cm}}

\begin{document}
\maketitle

\begin{abstract}
\noindent
We construct a short-time parametrix for the exterior Dirichlet wave equation on the complement of a solid torus. The construction involves freezing one coefficient of the conjugated Laplacian. This results in a family of radial Pöschl--Teller equations. We bound the $L^2$ error in terms of the energy of the initial data. We then give several examples of Gaussian wavepackets in toroidal coordinates. 
\end{abstract}

\textbf{keywords:} Spectral geometry, wave equations, parametrices, scattering theory\\

\textbf{MSC classes:} 35S05; 35L20; 35L05; 35A08

\section{Introduction}

The exterior Dirichlet Helmholtz equation does not separate in toroidal coordinates. After conjugating the Laplacian, we freeze the coefficient $N^2$ at a point. The angular Fourier modes then satisfy one-dimensional Pöschl--Teller equations. Their eigenfunctions give an explicit leading term for the wave propagator.

We find the first instance of the use of the Pöschl--Teller potential in quantum chemistry \cite{PoschlTeller1933}. It has has also been used to model black hole quasinormal modes \cite{FerrariMashhoon1984}, and scattering problems \cite{ruij}.
Our construction follows the mode expansion in Mie scattering for a ball \cite{FangWaters2025}.
For the exterior of a circular cylinder, Iandoli and Ivanovici \cite{IandoliIvanovici2024} used an angular Fourier series and an axial Fourier transform to express the resolvent in terms of Bessel and Hankel functions.
Keeler and Marzuola \cite{KeelerMarzuola2022} used angular eigenmodes to obtain resolvent and weighted dispersive estimates for Schr\"odinger operators with decaying radial potentials on product cones.

For strictly convex obstacles, Melrose and Taylor constructed parametrices using Fourier integral operators \cite{MelroseTaylor1985}. Here the radial Green's function enforces the boundary condition mode by mode, similar to Mie scattering. To describe our problem, we first consider: 

Let $\mathbf T^2$ be the open solid torus, $\mathbb T^2=\partial\mathbf T^2$ and $M=\R^3\setminus\mathbf T^2$. For the frozen operator $\Delta_{P,0}$ of Section~\ref{sec_Lap_tor}, we construct the Green's operator $\mathcal G_P$ for the Dirichlet problem
\begin{align}\label{GFT}
	\begin{aligned}
u_{tt}-\Delta_{P,0}u&=0\qquad &&\text{in}\quad (0,T)\times M,\\
u(0,\cdot)=f_0,\quad u_t(0,\cdot)&=f_1\qquad &&\text{in}\quad M,\\
u(t,\cdot)&=0 \quad &&\text{on} \quad (0,T)\times\mathbb{T}^2,
	\end{aligned}
\end{align}
For Cauchy data in $C^\infty_c(M)$ this gives the unique classical solution.

We can then express the kernel in terms of classical Legendre functions. Multiplying the frozen solution by the weight $W$ of Section~\ref{sec_Lap_tor} gives a short-time approximation to the solution of
\begin{align}\label{waveT}
	\begin{aligned}
u_{tt}-\Delta_{\mathbb{T}^2}u&=0\qquad &&\text{in}\quad (0,T)\times M,\\
u(0,\cdot)=Wf_0,\quad u_t(0,\cdot)&=Wf_1\qquad &&\text{in}\quad M,\\
u(t,\cdot)&=0 \quad &&\text{on} \quad (0,T)\times\mathbb{T}^2,
	\end{aligned}
\end{align}
on $0\leq t\leq T$.

\section{Toroidal coordinates}\label{toroidal1}
\noindent
In toroidal coordinates $(\phi_1,\phi_2,\tau)$ the Cartesian coordinates are given by \cite{MoonSpencer1988}
\begin{align}
	\begin{aligned}\label{toroidal}
	x(\phi_1,\phi_2,\tau;a) &= a \frac{\sinh \tau}{\cosh\tau-\cos\phi_1} \cos\phi_2,\\
	y(\phi_1,\phi_2,\tau;a) &= a \frac{\sinh \tau}{\cosh\tau-\cos\phi_1} \sin\phi_2,\\
	z(\phi_1,\phi_2,\tau;a) &= a \frac{\sin\phi_1}{\cosh\tau-\cos\phi_1}.
	\end{aligned}
\end{align}
Here, $\phi_1$ and $\phi_2$ lie in $(-\pi,\pi]$ and  $[0,2\pi)$, respectively, and $\tau>0$; the axis is obtained by coordinate limits. 
$a>0$ is a free parameter. The domains of $\phi_1$ and $\phi_2$ are interpreted as circles, which we denote by $\mathbb{S}^1_{\phi_1}$ and $\mathbb{S}^1_{\phi_2}$.

Setting $\phi_2=0$ in \eqref{toroidal} gives bipolar coordinates in the right half of the $(x,z)$-plane,

\begin{align*}
	\begin{aligned}
		x(\phi_1,0,\tau;a) &= a \frac{\sinh \tau}{\cosh\tau-\cos\phi_1},\\
		z(\phi_1,0,\tau;a) &= a \frac{\sin\phi_1}{\cosh\tau-\cos\phi_1}.
	\end{aligned}
\end{align*}
A calculation shows that 
\begin{equation*}
	(x-a\coth\tau)^2 + z^2 = \frac{a^2}{\sinh^2\tau} = a^2 \mathrm{csch}^2\tau.
\end{equation*}
For fixed $\tau>0$ this is a circle of radius $a\mathrm{csch}\tau$ centred at $(a\coth\tau,0)$.
Rotating the circle about the $z$-axis gives the torus in \eqref{toroidal}.

For comparison, the usual parametrization of a torus with minor radius $r$ and major radius $R$, $0<r<R$, is
\begin{align*}
	\begin{aligned}
		X(\theta,\varphi) &=(R + r\cos\theta) \cos\varphi,\\
		Y(\theta,\varphi) &= (R + r\cos\theta) \sin\varphi,\\
		Z(\theta,\varphi) &= r\sin\theta.
	\end{aligned}
\end{align*} 
By fixing
\begin{equation*}
	a =\sqrt{R^2-r^2}>0 \quad\text{and}\quad 
	\tau=\log\bigg(\frac{R + \sqrt{R^2-r^2}}{r}\bigg) > 0,
\end{equation*}
we obtain another parametrization of the same torus by means of the toroidal coordinates in \eqref{toroidal}. 

\subsection{Laplacian in toroidal coordinates} \label{sec_Lap_tor}
For fixed $a>0$, let $u\colon \bS^1_{\phi_1} \times \bS^1_{\phi_2} \times \R^+ \to \R$ be a $C^2$-function of the toroidal coordinates $(\phi_1,\phi_2,\tau)\in\bS^1_{\phi_1} \times \bS^1_{\phi_2} \times \R^+$. The Laplacian of $u$ is given by (cf.\@ \cite{MoonSpencer1988}) 
\begin{align} \label{toroidal_Laplace}
\begin{aligned}
	\Delta u = \frac{(\cosh\tau-\cos\phi_1)^3}{a^2\sinh\tau} 
	\bigg[ \partial_{\phi_1}\bigg(\frac{\sinh(\tau)\, u_{\phi_1}}{\cosh\tau-\cos\phi_1}\bigg) 
	+ \frac{\mathrm{csch}(\tau)\,u_{\phi_2\phi_2}}{\cosh\tau-\cos\phi_1}
	+ \partial_{\tau}\bigg(\frac{\sinh(\tau)\, u_{\tau}}{\cosh\tau-\cos\phi_1}\bigg)
	\bigg],
\end{aligned}
\end{align}
where subscripts denote the respective partial derivatives, e.g.\@ $u_{\phi_1}=\frac{\partial u}{\partial\phi_1}$.
The related P\"oschl-Teller operator $\Delta_{P}$ acts on $u$ as follows: 
\begin{align} \label{Delta_P}
\Delta_P u = \frac{(\cosh\tau-\cos\phi_1)^2}{a^2\sinh\tau} 
	\bigg[ \sinh(\tau)\, u_{\phi_1\phi_1} 
	+ \mathrm{csch}(\tau)\,\Big(u_{\phi_2\phi_2}+\frac{1}{4}u\Big)
	+ \sinh(\tau)\, u_{\tau\tau}
	\bigg].
\end{align}
Let 
\begin{equation*}
	N\coloneqq(\cosh\tau-\cos\phi_1)/a \quad\text{and}\quad W\coloneqq((\cosh\tau-\cos\phi_1)/\sinh\tau)^{1/2}.
\end{equation*}
A direct substitution shows that
\begin{align}\label{conjugation}
 W^{-1}\Delta(Wf)=\Delta_Pf.
\end{align}
We further set 
\begin{align}
N_0\coloneqq(\cosh\tau_0-\cos\phi_{1,0})/a
\end{align}
for fixed coordinates $\tau_0\in[0,\tau_1]$, $\phi_{1,0}\in(-\pi,\pi]$ with $(\tau_0,\phi_{1,0})\neq(0,0)$, where $\tau_1$ specifies the torus (cf.\@ Section \ref{sec_Obst_tor}), and define the ``frozen" operator $\Delta_{P,0}$ by
\begin{equation*}
	\Delta_{P,0}u\coloneqq \frac{N_0^2 }{\sinh\tau} 
	\bigg[ \sinh(\tau)\, u_{\phi_1\phi_1} 
	+ \mathrm{csch}(\tau)\,\Big(u_{\phi_2\phi_2}+\frac{1}{4}u\Big)
	+ \sinh(\tau)\, u_{\tau\tau}
	\bigg].
\end{equation*}

For a function $h$ of $(\phi_1,\phi_2,\tau)$, which is integrable in the angular variables $\phi=(\phi_1,\phi_2)$, we denote its Fourier coefficients by
\begin{align}\label{fourier_coeff}
	\hat{h}_{m,n}(\tau) = \frac{1}{(2\pi)^2}\int\limits_0^{2\pi}\int\limits_{-\pi}^{\pi}h(\theta_1,\theta_2,\tau)\, \e^{-i(m\theta_1+n\theta_2)}\,\dx \theta_1\,\dx \theta_2,
	\qquad h=\sum\limits_{m,n\in\mathbb{Z}}\hat{h}_{m,n}\,\e^{i(m\phi_1+n\phi_2)}.
\end{align}
Throughout, $m\in\mathbb{Z}$ denotes the Fourier variable dual to $\phi_1$ and $n\in\mathbb{Z}$ the one dual to $\phi_2$; we further write $\mu\coloneqq|n|$.
Taking Fourier coefficients of $\Delta_{P,0}u$ results in
\begin{align}\label{angular_transform}
\widehat{(\Delta_{P,0}u)}_{m,n}=-\frac{N_0^2}{\sinh\tau}
\bigg[\sinh(\tau)\,m^2\,\hat u_{m,n}
+\mathrm{csch}(\tau)\,\Big(n^2-\frac14\Big)\hat u_{m,n}
-\sinh(\tau)\,\partial_\tau^2\hat u_{m,n}\bigg].
\end{align}
$\Delta$ and $\Delta_P$ have the same principal symbol and that the symbol of $\Delta_P$ differs from the corresponding principal symbol only by a zero order term.  

\subsection{Formulation of the obstacle problem} \label{sec_Obst_tor}

Fix $a,\tau_1>0$. The surface $\tau=\tau_1$ is the torus
\begin{equation*}
	\mathbb{T}^2=\big\{\big(x,y,z\big)(\phi_1,\phi_2,\tau_1;a)\colon \phi_1\in (-\pi,\pi],\phi_2\in[0,2\pi)\big\}.
\end{equation*} 
Furthermore, we denote the corresponding open solid torus by $\mathbf{T}^2$ and refer to the closed subset
\begin{equation*}
	M=\mathbb{R}^3\setminus \mathbf{T}^2  = \overline{\big\{\big(x,y,z\big)(\phi_1,\phi_2,\tau;a)\colon \phi_1\in (-\pi,\pi] ,\phi_2\in[0,2\pi),\, \tau\in(0,\tau_1]\big\}}\subseteq\mathbb{R}^3
\end{equation*} 
as the exterior of $\mathbf{T}^2$. Both $\mathbf{T}^2$ and $M$ have $\mathbb{T}^2$ as their boundary: $\partial\mathbf{T}^2=\partial M=\mathbb{T}^2$.
We denote by $\Delta_{\mathbb{T}^2}$ the Laplacian $\Delta$ with vanishing Dirichlet boundary data on $\mathbb{T}^2$.

The exterior Helmholtz equation does not separate in these coordinates \cite{weston58}.
Freezing $N$ at $N_0$ gives the Dirichlet problem
\begin{align}\label{CFPT}
	\begin{aligned}
		u_{t t} - \Delta_{P,0}u = 0 \qquad & \text{for} \quad  (t,\phi_1,\phi_2,\tau) \in (0,\infty)\times (-\pi,\pi] \times [0,2\pi) \times (0,\tau_1], \\
		u(t,\phi_1,\phi_2,\tau) = 0 \qquad & \text{for} \quad  (t,\phi_1,\phi_2,\tau)\in (0,\infty)\times (-\pi,\pi]\times [0,2\pi) \times \{\tau_1\}, \\
		\begin{cases}
		u(0,\phi_1,\phi_2,\tau) \;= u_0(\phi_1,\phi_2,\tau), \\  
		u_t(0,\phi_1,\phi_2,\tau) = u_1(\phi_1,\phi_2,\tau),
		\end{cases}
		 \qquad & \text{for} \quad  (\phi_1,\phi_2,\tau)\in (-\pi,\pi]\times [0,2\pi) \times (0,\tau_1].		
	\end{aligned}
\end{align}
By \eqref{conjugation}, the solution with $\Delta_P$ in place of $\Delta_{P,0}$ gives, after multiplication by $W$, the physical wave with data $(Wu_0,Wu_1)$. For the frozen problem \eqref{CFPT}, Section~\ref{sec_kernel} gives the solution operator and Section~\ref{errorestimates} bounds the error.

\section{Review of the solutions to the P\"oschl-Teller Equation}\label{ptreview}
\noindent
We use $x>0$ for the radial variable in the separated equation.

Let $\mu\geq 0$ be a parameter and $V\colon(0,\infty)\to\R$ be a potential defined by 
\begin{align}
V(x)=\Big(\mu^2-\frac{1}{4}\Big)\mathrm{sinh}^{-2}(x).
\end{align}
We consider the following P\"oschl-Teller eigenfunction equation:
\begin{align}\label{stationarys}
-\frac{\dx^2}{\dx x^2}\psi(x)+V(x)\psi(x)=k^2\psi(x), \quad k^2>0,
\end{align}
which takes the form of a stationary Schr\"odinger equation.
Let $\mathsf{P}^{q}_{n}(x)$ denote the Legendre function of order $q$ and degree $n$ on $x>1$, cf.\@ \cite[\S14.3(ii)]{NIST}. 
The solution with Friedrichs behaviour at the origin is
\begin{align}\label{conical_P}
\sqrt{\sinh(x)}\mathsf{P}^{-\mu}_{ik-\frac{1}{2}}(\cosh x)=\frac{(\sinh x)^{\mu+\frac{1}{2}}}{2^{\mu}} \, \mathbf{F}\bigg(\frac{\mu+\frac{1}{2}+ik}{2},\frac{\mu+\frac{1}{2}-ik}{2},\mu+1,-\sinh^2(x)\bigg).
\end{align}
This is a consequence of \cite[(14.3.15)]{NIST} and a quadratic transformation, as remarked in \cite[(4.3--4.5)]{ruij}. Olver's hypergeometric function $\mathbf{F}={}_2F_1/\Gamma(c)$ is well-defined for $k,\mu>0$. Indeed, following \cite[Section 2]{daSilva_etal2024} and substituting $\xi =\tanh^2(x)$ in \eqref{stationarys} yields 
\begin{align}\label{eq_xi}
\xi(1-\xi)\frac{\dx^2\psi(\xi)}{\dx\xi^2}+\left(\frac{1}{2}-\frac{3}{2}\xi\right)\frac{\dx\psi(\xi)}{\dx\xi}+\left(\frac{1}{(1-\xi)}\frac{k^2}{4}-\frac{1}{\xi}\frac{\mu^2-\frac{1}{4}}{4}\right)\psi(\xi)=0.
\end{align}
For $\mu\notin\mathbb{Z}$, the general solution of this equation can be written as
\begin{align} \label{psi_k}
\psi_k(x)=\left(\tanh(x)\right)^{-\mu+\frac{1}{2}}(\cosh(x))^{-ik}F(x),
\end{align}
where $F(x)$ is a combination of hypergeometric functions: 
\begin{align}
\begin{aligned}	\label{hyper}
F(x)&=A\,\mathbf{F}\!\left(\frac{1}{2}\Big(1-\Big(\mu+\frac{1}{2}\Big)+ik\Big),\frac{1}{2}\Big(-\Big(\mu+\frac{1}{2}\Big)+2+ik\Big),1-\mu,\mathrm{tanh}^2(x)\right)\\ &\quad +B\tanh^{2\mu}(x)\,\mathbf{F}\!\left(\frac{1}{2}\Big(\mu+\frac{1}{2}+ik\Big),\frac{1}{2}\Big(\mu+\frac{1}{2}+1+ik\Big),\mu+1,\mathrm{tanh}^2(x)\right)
\end{aligned}
\end{align}
with constants $A$ and $B$.  
For $a,c\in\C$ and $0<z<1$, we use the identity \cite[(15.9.17)]{NIST}

\begin{equation*}
	\mathbf{F}\!\left(a,a+\frac{1}{2},c;z\right)=2^{c-1}z^{\frac{1-c}{2}}%
	(1-z)^{-a+\frac{c-1}{2}}\,\mathsf{P}^{1-c}_{2a-c}\left(\frac{1}{\sqrt{1-z}}\right),
\end{equation*}
in the $B$-term of \eqref{hyper}. Choosing $B=2^{-\mu}$ and $A=0$ gives \eqref{conical_P}.
This solution is asymptotic to $x^{\mu+1/2}/(2^\mu\Gamma(\mu+1))$ at the origin and satisfies the Friedrichs condition used in Section~\ref{sec_kernel}. For integer $\mu$ the two terms in \eqref{hyper} are linearly dependent. In the torus problem $\mu=|n|$.
%

\section{The kernel representations}\label{sec_kernel}
We define the non-negative self-adjoint operator
\begin{equation*}
	A\coloneqq-\partial_{\tau}^2-\partial_{\phi_1}^2-\mathrm{csch}^2\tau\left(\partial_{\phi_2}^2+\frac{1}{4}\right),
\end{equation*}
which satisfies $-\Delta_{P,0}=N_0^2A$. 
It is associated to the quadratic form 
\begin{align}
q(f,f) \coloneqq \Big\| \partial_{\tau}f-\frac{f}{2\tau}\Big\|_2^2+\|\partial_{\phi_1}f\|_2^2+\|\mathrm{csch}(\tau)\,\partial_{\phi_2}f\|_2^2+\frac{1}{4}\int\limits_{M}(\tau^{-2}-\mathrm{csch}^2\tau)|f|^2\,\dx X\geq 0,
\end{align} 
where $\dx X\coloneqq \dx\tau \, \dx\phi_1 \, \dx\phi_2$ denotes the volume element in toroidal coordinates. 
The operator $A$ is the Friedrichs extension associated to this quadratic form defined on $C_c^{\infty}((0,\tau_1)\times\mathbb{T}^2)\subset L^2(\dx\tau,\dx\phi_1,\dx\phi_2)$. 
We use the Fourier coefficients $\hat f_{m,n}$ from \eqref{fourier_coeff}, with $m$ dual to $\phi_1$, $n$ dual to $\phi_2$ and $\mu=|n|$.
On the $(m,n)$-th Fourier mode, $A=P_{\mu,D}+m^2$ with $\mu=|n|$, where $P_{\mu,D}$ is the Pöschl-Teller operator
\begin{align}
P_{\mu,D}\coloneqq -\frac{\dx^2}{\dx\tau^2}+\frac{\mu^2-\frac{1}{4}}{\sinh^2\tau} \qquad\text{for} \quad 0<\tau<\tau_1,
\end{align}
acting on functions $v$ obeying the Dirichlet boundary condition $v(\tau_1)=0$.
Its Dirichlet eigenfunctions $v$ satisfy $P_{\mu,D}v=k^2v$ and $v(\tau_1)=0$.
Thus, 
\begin{align}
-\Delta_{P,0}(\e^{i(m\phi_1+n\phi_2)}v)=N_0^2(k^2+m^2)\e^{i(m\phi_1+n\phi_2)}v \quad\mathrm{whenever}\quad P_{|n|,D}v=k^2v.
\end{align} 
For $\Im k\geq0$, set $\nu\coloneqq-\frac{1}{2}-ik$ and take the two radial solutions
\begin{align}
p_{k,\mu}(\tau)=\sqrt{\sinh\tau}\,\mathsf{P}_{\nu}^{-\mu}(\cosh\tau), \qquad q_{k,\mu}(\tau)=\sqrt{\sinh \tau}\,\boldsymbol{Q}_{\nu}^{\mu}(\cosh\tau),
\end{align}
where $\boldsymbol{Q}^{\mu}_{\nu}$ is Olver's associated Legendre function of the second kind, cf.\@ \cite[(14.3.10)]{NIST}. Since $\mathsf{P}^{-\mu}_{-\nu-1}=\mathsf{P}^{-\mu}_{\nu}$ by \cite[(14.9.11)]{NIST}, $p_{k,\mu}$ coincides with the function in \eqref{conical_P}. By \cite[(14.2.8)]{NIST}, the Wronskian is given by
\begin{align}\label{wronskian}
w_{k,\mu}=p_{k,\mu}q'_{k,\mu}-p_{k,\mu}'q_{k,\mu}=-\frac{1}{\Gamma(\frac{1}{2}-ik+\mu)},
\end{align}
which is finite and nonzero for all $\Im k\geq0$ and $\mu\geq0$.
Subtracting a multiple of $p_{k,\mu}$ from $q_{k,\mu}$ gives a solution vanishing at $\tau_1$.
\begin{proposition}
Let $\Im k\geq0$ with $k^2 \in\C\setminus \operatorname{spec}(P_{\mu,D})$. Then the kernel of the resolvent $(P_{\mu,D}-k^2)^{-1}$ admits the following representation
\begin{align}
G_{\mu}(k,\tau,\sigma) \coloneqq\begin{cases}
-\frac{1}{w_{k,\mu}}\left(q_{k,\mu}(\tau)p_{k,\mu}(\sigma)-\frac{q_{k,\mu}(\tau_1)}{p_{k,\mu}(\tau_1)}p_{k,\mu}(\tau)p_{k,\mu}(\sigma)\right), \quad \tau>\sigma, \\
-\frac{1}{w_{k,\mu}}\left(p_{k,\mu}(\tau)q_{k,\mu}(\sigma)-\frac{q_{k,\mu}(\tau_1)}{p_{k,\mu}(\tau_1)}p_{k,\mu}(\tau)p_{k,\mu}(\sigma)\right),  \quad \tau<\sigma.
\end{cases} 
\end{align}
Let $s\in \mathbb{C}$ with $\Re s>0$ and set $k_m(s)\coloneqq i\sqrt{m^2+s^2/N_0^2}$, where the square root is the principal branch, so that $\Im k_m(s)>0$ and $k_m(s)^2=-m^2-s^2/N_0^2$. Then for $X\coloneqq(\tau,\phi_1,\phi_2)$ and $Y\coloneqq(\sigma,\theta_1,\theta_2)$, 
\begin{align}
\mathcal{G}(s,X,Y) \coloneqq \frac{1}{(2\pi)^2N_0^2}\sum\limits_{m,n\in\mathbb{Z}}\e^{i(m(\phi_1-\theta_1)+n(\phi_2-\theta_2))}G_{|n|}(k_m(s),\tau,\sigma) 
\end{align}
represents the kernel of $(-\Delta_{P,0}+s^2)^{-1}$. 
\end{proposition}
\begin{proof}
Each branch satisfies the radial equation away from $\tau=\sigma$, and the boundary correction gives
\begin{align}
G_{\mu}(k,\tau_1,\sigma)= 0, \qquad P_{\mu,D}G_{\mu}(k,\tau,\sigma)=k^2G_{\mu}(k,\tau,\sigma)\quad\text{for}\quad \sigma\neq \tau,
\end{align}
and $G_\mu(k,\cdot,\sigma)$ is a multiple of $p_{k,\mu}$ near $\tau=0$, where $p_{k,\mu}(\tau_1)\neq0$ since $k^2$ is not a Dirichlet eigenvalue. By \eqref{wronskian}, we also have 
\begin{align}
\partial_{\tau}G_{\mu}(k,\sigma^+,\sigma)-\partial_{\tau}G_{\mu}(k,\sigma^-,\sigma)=-1. 
\end{align} 
Thus $(P_{\mu,D}-k^2)G_{\mu}(k,\cdot,\sigma)=\delta_{\sigma}$.
Summing the angular modes gives the second formula, since
\begin{align}
N_0^2(P_{|n|,D}+m^2)+s^2=N_0^2(P_{|n|,D}-k_m(s)^2). 
\end{align}
\end{proof}
Since $P_{\mu,D}\geq0$, $k^2$ lies in the resolvent set whenever $\Im k>0$. This applies to $k_m(s)$ for $\Re s>0$.
\begin{proposition}
Let $f_0,f_1\in C_c^{\infty}(M)$ and $\gamma>0$. Then
\begin{align}
v(X,t) \coloneqq \frac{1}{2\pi i}\int\limits_{\gamma-i\infty}^{\gamma+i\infty}\int\limits_{M}\e^{st}\mathcal{G}(s,X,Y)(sf_0(Y)+f_1(Y))
\,\dx Y \, \dx s 
\end{align}
solves the problem \eqref{GFT} in toroidal coordinates. 
\end{proposition} 
\begin{proof}
Taking the Laplace transform in time in \eqref{GFT} gives
\begin{align}
\hat{v}(s)=(s^2+N_0^2A)^{-1}(sf_0+f_1), \quad \Re s>0. 
\end{align} 
Inserting the kernel $\mathcal G$ of the resolvent and inverting the Laplace transform on $\Re s=\gamma$ gives the formula for $v$.
\end{proof} 
\section{Error estimates}\label{errorestimates}

We would like to demonstrate that the error between \eqref{GFT} and \eqref{waveT} can be quantified. In particular we know that $w=u-Wv$ as $0$ Cauchy data. 

We now let the frozen equation have propagation $N_0$. Then we consider $K_0\subset [0,\tau_1]\times \mathbb{S}_{\phi_1}$ as the projection of 
 of $\supp f_0\cup\supp f_1$, and let $K_t$ be the set of points of $[0,\tau_1]\times\mathbb S^1_{\phi_1}$, whose distance is at most $N_0t$ away from $K_0$, so that $v(t)$ is supported in $K_t\times\mathbb S^1_{\phi_2}$. We recall that the point $(\tau,\phi_1)=(0,0)$ is actually spatial infinity. Then the error we get by freezing $N$ is 
\[
\eta \coloneqq 1-\frac{N_0^2}{N^2},\qquad \alpha_t\coloneqq\sup_{K_t}|\eta|,\qquad \beta_t\coloneqq\sup_{K_t}|\nabla \eta|,
\]
where $\nabla=(\partial_\tau,\partial_{\phi_1})$. We can then write  $N_{\max}\coloneqq\max N=(\cosh\tau_1+1)/a$ and we set $D\coloneqq\partial_\tau-\frac{1}{2\tau}$. We can now quantify the error as follows: 
\begin{proposition}\label{errorprop}
Let $f_0,f_1\in C_c^\infty((0,\tau_1)\times\mathbb S^1_{\phi_1}\times\mathbb S^1_{\phi_2})$, let $v$ solve \eqref{GFT} and $u$ solve \eqref{waveT}. Let $T>0$ be such that $(0,0)\notin K_T$. Then for $0\leq t\leq T$ we have that 
\begin{align}\label{error_est}
\|u(t)-Wv(t)\|_{L^2(M)}\leq\Big(t\,\alpha_t+\frac{N_{\max}t^2}{2}\,\beta_t\Big)\Big(\|\nabla(Wf_0)\|_{L^2(M)}+\frac{N_{\max}}{N_0}\|Wf_1\|_{L^2(M)}\Big).
\end{align}
\end{proposition}
\begin{proof}
Because $N>0$ on $K_T$, $Wv(s)$ lies in the domain of $\Delta_{\mathbb T^2}$ for $0\leq s\leq T$. Using equation \eqref{conjugation} and the definition $\Delta_P=-N^2A$, the error $w=u-Wv$ solves
\[
 (\partial_s^2-\Delta_{\mathbb T^2})w=-R,\qquad R=W\big((N^2-N_0^2)Av\big),\qquad w(0)=\partial_sw(0)=0. 
\]
Therefore we have by Duhamel's principle
\[
 w(t)=-\int_0^t\frac{\sin\big((t-s)\sqrt{-\Delta_{\mathbb T^2}}\big)}{\sqrt{-\Delta_{\mathbb T^2}}}R(s)\,\dx s.
\]
Whenever $z\in C_c^\infty((0,\tau_1)\times\mathbb T^2)$, we then have $\langle Wf,Wz\rangle_{L^2(M)}=a\langle N^{-2}f,z\rangle$ and $\|\nabla(Wz)\|^2_{L^2(M)}=a\,q(z)$. Because $v(s)$ is supported in $K_s\times\mathbb S^1_{\phi_2}$, whenever $\eta$ is smooth and independent of $\phi_2$, we then obtain
\[
 \langle R,Wz\rangle_{L^2(M)}=a\langle \eta Av,z\rangle=a\,q(v,\eta z)=a\,q_\eta(v,z)+a\big\langle(\partial_\tau \eta)Dv+(\partial_{\phi_1}\eta)\partial_{\phi_1}v,\,z\big\rangle.
\]
Here we have
\[
 q_\eta(v,z)=\langle \eta Dv,Dz\rangle+\langle \eta\partial_{\phi_1}v,\partial_{\phi_1}z\rangle+\langle \eta\,\mathrm{csch}(\tau)\partial_{\phi_2}v,\mathrm{csch}(\tau)\partial_{\phi_2}z\rangle+\tfrac14\langle \eta(\tau^{-2}-\mathrm{csch}^2\tau)v,z\rangle
\]
satisfies $|q_\eta(v(s),z)|\leq\alpha_s\sqrt{q(v(s))}\sqrt{q(z)}$. Now, $R=R_1+R_2$ with
\[
 R_2=W\big(N^2((\partial_\tau \eta)Dv+(\partial_{\phi_1}\eta)\partial_{\phi_1}v)\big),
\]
and
\begin{align*}
 \|(-\Delta_{\mathbb T^2})^{-1/2}R_1(s)\|_{L^2(M)}&\leq\sqrt a\,\alpha_s\sqrt{q(v(s))},\\
 \|R_2(s)\|_{L^2(M)}&\leq\sqrt a\,N_{\max}\beta_s\sqrt{q(v(s))}.
\end{align*}
Because we have that 
 $\|\sin(r\sqrt{-\Delta_{\mathbb T^2}})(-\Delta_{\mathbb T^2})^{-1/2}g\|\leq\min\{\|(-\Delta_{\mathbb T^2})^{-1/2}g\|,\,r\|g\|\}$, it follows that
\[
 \|w(t)\|_{L^2(M)}\leq\sqrt a\Big(t\,\alpha_t+\frac{N_{\max}t^2}{2}\,\beta_t\Big)\sup_{0\leq s\leq t}\sqrt{q(v(s))}.
\]
We now use the spectral theorem to obtain
\[
 \sqrt{q(v(s))}=\|A^{1/2}v(s)\|_2\leq\sqrt{q(f_0)}+N_0^{-1}\|f_1\|_2,
\]
and $\sqrt a\sqrt{q(f_0)}=\|\nabla(Wf_0)\|_{L^2(M)}$, $\sqrt a\|f_1\|_2\leq N_{\max}\|Wf_1\|_{L^2(M)}$.
\end{proof}

\section{A wave packet in toroidal coordinates}\label{sec:wavepacket}

The goal of this section is to give an example. First we find a $\chi\in C_c^\infty((0,\infty))$ and $\psi\in C^\infty(\mathbb T^2)$ supported where $\operatorname{dist}_{\mathbb T}(\phi_1,0)\geq\delta$, for some $0<\delta<\pi$. The constants $s_0,k\in\mathbb R$ and $\sigma>0$ are considered fixed. We then take $0<h\leq1$ sufficiently small such that $h\max(\supp\chi)<\tau_1$, and then we set
\begin{align}
G(\rho)=\chi(\rho)\e^{\left(-\frac{(\frac{1}{\rho}-s_0)^2}{2\sigma^2}\right)}\e^{\frac{ik}{\rho}},\qquad F(\rho)=-G'(\rho)-\frac{G(\rho)}{2\rho},
\end{align}
and
\begin{align}
D&=\partial_{\tau}-\frac{1}{2\tau},\quad D^*=-\partial_{\tau}-\frac{1}{2\tau},\\
g_h&=c_h\sqrt{h}\,G\Big(\frac{\tau}{h}\Big)\psi(\phi).
\end{align}
Using the initial data choice $(0,f_1)=(0,N^2D^*g_h)$, we can normalized by a constant, $c_h>0$, so that $\|Wf_1\|_{L^2(M)}=1$.
To see this we set $f_1=c_hh^{-1/2}N^2F(\tau/h)\psi$, so
\[
 c_h^{-2}=a\int N(h\rho,\phi)^2|F(\rho)\psi(\phi)|^2\,\dx\rho\,\dx\phi.
\]
Thus $c_h$ is bounded above and below independently of $h$, since the integral is obviously finite. 
It follows that 
\begin{align}
\|g_h\|_{L^2}=c_hh\|G\psi\|_{L^2}=O(h).
\end{align}
Now of course we see that 
\begin{align}
\|\partial_tu(0)\|_{L^2(M)}=\|WN^2D^*g_h\|_{L^2(M)}=1.
\end{align}
Because $\langle Wf_1,Wz\rangle_{L^2(M)}=a\langle g_h,Dz\rangle$ and $\|Dz\|_2\leq\sqrt{q(z)}$, the exact solution satisfies
\[
\|u(t)\|_{L^2(M)}\leq\sqrt a\,\|g_h\|_2=O(h)\qquad\text{for all }t.
\]

The frozen equation has propagation speed $N_0$ in $(\tau,\phi_1)$. Let $K_t$ be the set of points of $[0,\tau_1]\times\mathbb S^1_{\phi_1}$ at distance at most $N_0t$ from the projection of $\supp f_1$, so that $v(t)$ is supported in $K_t\times\mathbb S^1_{\phi_2}$. For $t\leq\delta/(2N_0)$ we have $\operatorname{dist}_{\mathbb T}(\phi_1,0)\geq\delta/2$ on $K_t$, hence $N\geq N_*:=(1-\cos\frac\delta2)/a$ there.
\begin{corollary}\label{packet_cor}
Let $u$ be the exact wave with initial data $(0,WN^2D^*g_h)$, and let $v$ be the frozen wave with initial data $(0,N^2D^*g_h)$. There exists a constant $C$, independent of $h$ and $t$, such that for $0\leq t\leq\delta/(2N_0)$
\begin{align}\label{packet_error}
\|u(t)-Wv(t)\|_{L^2(M)}\leq C\Big(t\sup_{K_t}|N-N_0|+t^2\sup_{K_t}|\nabla N|\Big).
\end{align}
\end{corollary}
\begin{proof}
The data $(0,f_1)$ with $f_1=N^2D^*g_h$ satisfy the assumptions of Proposition \ref{errorprop} with $T=\delta/(2N_0)$, since $N\geq N_*>0$ on $K_T$. As $\|Wf_1\|_{L^2(M)}=1$, the estimate \eqref{error_est} reads
\[
 \|u(t)-Wv(t)\|_{L^2(M)}\leq\frac{N_{\max}}{N_0}\Big(t\sup_{K_t}|\eta|+\frac{N_{\max}t^2}{2}\sup_{K_t}|\nabla \eta|\Big),\qquad \eta=1-\frac{N_0^2}{N^2}.
\]
Since $N_*\leq N\leq N_{\max}$ on $K_t$, we have $|\eta|=(N+N_0)|N-N_0|/N^2\leq C|N-N_0|$ and $|\nabla \eta|=2N_0^2|\nabla N|/N^3\leq C|\nabla N|$ there, with $C$ depending only on $N_0$, $N_*$ and $N_{\max}$, which proves \eqref{packet_error}.
\end{proof}
In particular $\|u(t)-Wv(t)\|_{L^2(M)}\leq Ct$. Since $\|\partial_tu(0)\|_{L^2(M)}=1$, $u(t)$ itself has size comparable to $t$ for $t\ll h$, and $\|u(t)\|_{L^2(M)}=O(h)$ for all $t$. To improve the error bound one would need $\sup_{K_t}|N-N_0|$ and $\sup_{K_t}|\nabla N|$ in \eqref{packet_error} to be small, which is achieved by placing the packet where $N$ is almost constant. 

\subsection{A packet at the centre of the torus}
The point $(\tau,\phi_1)=(0,\pi)$ is the origin of $\mathbb R^3$ and the center of the hole of the torus, where $\nabla N=0$ and $N=2/a$. As such we select $N_0=2/a$, $0<h\leq\frac12$, and replace $\psi$ by
\[
\psi_h(\phi)=h^{-1/4}\,\Psi\big(h^{-1/2}(\phi_1-\pi),\phi_2\big),\qquad \Psi\in C_c^\infty((-1,1)\times\mathbb S^1),
\]
leaving the rest of the construction unchanged. Then $\psi_h$ is supported where $\operatorname{dist}_{\mathbb T}(\phi_1,0)\geq\pi/2$, and $c_h$ is still bounded above and below, so Corollary \ref{packet_cor} applies with $\delta=\pi/2$.
\begin{corollary}
With $\psi_h$ and $N_0=2/a$ as above,
\[
\|u(t)-Wv(t)\|_{L^2(M)}\leq C\big(ht+t^3\big),\qquad 0\leq t\leq\frac{\pi a}{8}.
\]
In particular, $\|u(t)-Wv(t)\|_{L^2(M)}\leq Ch(\theta h^{1/3}+\theta^3)$ for $0\leq t\leq\theta h^{1/3}\leq\pi a/8$, so the error is small compared with the $O(h)$ bound for $\|u(t)\|_{L^2(M)}$, for $0\leq t\leq\theta h^{1/3}$ with $\theta$ small.
\end{corollary}
\begin{proof}
On $K_t$ we have $\tau\leq h\max(\supp\chi)+N_0t$ and $|\phi_1-\pi|\leq h^{1/2}+N_0t$. Since $\cosh\tau-1\leq\frac12\tau^2\cosh\tau_1$, $1+\cos\phi_1\leq\frac12(\phi_1-\pi)^2$ and $|\nabla N|=a^{-1}(\sinh^2\tau+\sin^2\phi_1)^{1/2}$,
\[
 \sup_{K_t}|N-N_0|\leq C(h+t^2),\qquad \sup_{K_t}|\nabla N|\leq C(h^{1/2}+t).
\]
By \eqref{packet_error}, $\|u(t)-Wv(t)\|_{L^2(M)}\leq C(ht+t^3+h^{1/2}t^2)$, and $h^{1/2}t^2\leq\frac12(ht+t^3)$.
\end{proof}

\textbf{Acknowledgments:} The authors wish to thank an anonymous referee for comments on an earlier version of this paper. The author A.W. wishes to thank Jeremy Marzuola for pointing us in the direction of the example. The authors acknowledge discussions with ChatGpt 6 astra which helped with the details of the two examples.

\end{document}